\documentclass[12pt]{amsart}
\usepackage[cp1250]{inputenc}
\usepackage[T1]{fontenc}
\usepackage{amscd}

\theoremstyle{plain}
\newtheorem{thm}{Theorem}[section]
\newtheorem{lem}[thm]{Lemma}
\newtheorem{prop}[thm]{Proposition}
\newtheorem{cor}[thm]{Corollary}

\theoremstyle{definition}
\newtheorem{dff}[thm]{Definition}
\newtheorem{rem}[thm]{Remark}

\numberwithin{equation}{section}
\def\a{\alpha}

\def\A{\mathcal{A}}

\def\D{\mathcal{D}}

\def\F{\mathcal{F}}

\def\I{\mathcal{I}}
\def\U{\mathcal{U}}
\def\V{\mathcal{V}}

\def\ld{,\ldots,}

\def\phi{\varphi}

\def\ci{\circ}

\def\rz{\mathbb{R}}
\def\R{\rz}

\def\wyr#1{\textit{#1}}

\def\s{\subset}
\def\t{\times}
\def\r{\rightarrow}

 \DeclareMathOperator{\st}{star}

 \DeclareMathOperator{\diff}{Diff}
 \DeclareMathOperator{\id}{id}
 
\DeclareMathOperator{\supp}{supp} 
\DeclareMathOperator{\dom}{dom} 

\keywords{Diffeomorphism group, singular foliation, commutator
group, simple group,  leaf preserving diffeomorphism.}
\subjclass{53C12, 57R50, 57R52}
\thanks{Partially supported by the Polish Ministry of Science and Higher Education and the
AGH grant n. 11.420.04}
\address{Faculty of Applied Mathematics, AGH University of Science and
\linebreak Technology, al. Mickiewicza 30, 30-059 Krak\'ow,
Poland} \email{tomasz@agh.edu.pl}
\date{December 19, 2010}

\title{Correspondence between diffeomorphism groups and singular foliations }
\author{ Tomasz Rybicki}

\begin{document}

\maketitle

\begin{abstract} It is well-known that any isotopically connected
diffeomorphism group $G$ of a manifold  determines uniquely a
singular foliation $\F_G$. A one-to-one correspondence between the
class of singular foliations  and a subclass of diffeomorphism
groups is established.  As an illustration of this correspondence
it is shown that the commutator subgroup $[G,G]$ of an
isotopically connected, factorizable and non-fixing
$C^r$-diffeomorphism group $G$ is simple iff the foliation
$\F_{[G,G]}$ defined by $[G,G]$ admits no proper minimal sets. In
particular, the compactly supported $e$-component of the leaf
preserving $C^{\infty}$-diffeomorphism group of a regular
foliation $\F$ is simple iff $\F$ has no proper minimal sets.
 \end{abstract}

\section{Introduction}
Throughout by a \emph{foliation} we mean a singular foliation
(Sussmann \cite{sus}, Stefan \cite{st0}), and by a \emph{regular
foliation} we mean a foliation whose leaves have the same
dimension. Introducing the notion of foliations, Sussmann and
Stefan emphasized that they play a role of collections of
"accessible" sets. Alternatively, they regarded
 foliations as integrable smooth distributions. Another point of
view is to treat  foliations as by-products of non-transitive
geometric structures, c.f. \cite{dlm}, \cite{va} and examples in
\cite{ry2}. In Molino's approach some types of singular foliations
constitute collections of closures of leaves of certain regular
foliations (\cite{mol}, \cite{wol}). In this note we regard
foliations as a special type of diffeomorphism groups.

 Given a $C^{\infty}$ smooth paracompact boundaryless
manifold $M$, $\diff^r(M)_0$ (resp. $\diff^r_c(M)_0$), where
$1\leq r\leq\infty$,  is the subgroup of the group of all $C^r$
diffeomorphisms $\diff^r(M)$ on $M$ consisting of diffeomorphisms
that can be joined to the identity through a $C^r$ isotopy (resp.
compactly supported $C^r$ isotopy) on $M$. A diffeomorphism group
$G\leq \diff^r(M)$, is called \emph{isotopically connected}
 if any
element $f$ of $G$ can be joined to $\id_M$ through a $C^r$
isotopy in $G$. That is, there is a mapping $\R\t M\ni(t,x)\mapsto
f_t(x)\in M$ of class $C^r$ with $f_t\in G$ for all $t$ and such
that $f_0=\id$  and $f_1=f$. It is well-known that any
isotopically connected group $G\leq \diff^r(M)_0$ defines uniquely
a foliation of class $C^r$, designated by $\F_G$ (see sect. 2).

Our first aim is to establish a correspondence between the class
$\frak F^r(M)$ of all $C^r$-foliations on $M$ and a subclass of
diffeomorphism groups on $M$, and, by using it, to interpret some
results and some open problems concerning non-transitive
diffeomorphism groups.  The second aim is to prove new results
(Theorems 1.1 and 1.2) illustrating this correspondence.

A group  $G\leq\diff^r(M)$  is called \emph{factorizable} if for
every open cover $\U$ and every $g\in G$ there are $g_1\ld g_r\in
G$ with $g=g_1\ldots g_r$ and such that $g_i\in G_{U_i}$, $i=1\ld
r$, for some $U_1\ld U_r\in\U$. Here for $U\s M$ and
$G\leq\diff^r(M)$, $G_U$ stands for the identity component of the
group of all diffeomorphisms from $G$ compactly supported in $U$.
Next $G$ is said to be \wyr{non-fixing} if $G(x)\neq \{ x \}$ for
every $x \in X$.

\begin{thm}
 Assume that $G\leq
\diff^r_c(M)_0$, $ 1\leq r\leq\infty$, is isotopically connected,
non-fixing and factorizable group of diffeomorphisms of smooth
manifold $M$. Then the commutator group $[G,G]$ is simple if and
only if the corresponding foliation $\F_{[G,G]}$ admits no proper
(i.e. not equal to $M$) minimal set.
\end{thm}

In early 1970's  Thurston and Mather proved that  the group
$\diff^r_c(M)_0$, where $1\leq r\leq\infty$, $r\neq \dim(M)+1$, is
perfect and simple (see \cite{Thu74},\cite{Mat}, \cite{Ban97}).
Next, similar results were proved for classical diffeomorphism
groups of class $C^{\infty}$ (\cite{Ban97},
 \cite{ry5}).
  For  the significance of these
simplicity theorems, see, e.g., \cite{Ban97}, \cite{ry5} and
references therein.

Let $(M_i,\mathcal F_i)$, $i=1,2$, be foliated manifolds. A map
$f:M_1\rightarrow M_2$ is called \emph{foliation preserving}
if~$f(L_x)=L_{f(x)}$ for any $x\in M_1$, where $L_x$ is the leaf
meeting $x$. Next, if~$(M_1,\mathcal F_1)=(M_2,\mathcal F_2)$ then
$f$ is \emph{leaf preserving} if~$f(L_x)=L_x$ for all $x\in M_1$.
Throughout  $\diff^r(M,\mathcal F)$ will stand for the group
of~all leaf preserving $C^r$-diffeomorphisms of~a~foliated
manifold $(M,\mathcal F)$. Define $\diff^r(M,\mathcal F)_0$ and
$\diff^r_c(M,\mathcal F)_0$ analogously as above. Observe that a
perfectness theorem for the compactly supported identity component
$\diff^{\infty}_c(M,\mathcal F)_0$, being a non-transitive
 counterpart of Thurston's theorem,
 has been proved by the author in
 \cite{ry1} and by Tsuboi in \cite{Tsu1}. Next, the author in \cite{ry2},
 following Mather \cite{Mat}, II,
 showed that $\diff^r_c(M,\mathcal F)_0$ is perfect provided
 $1\leq r\leq\dim\F$. Observe that, in general, the group $\diff^r_c(M,\mathcal F)_0$
is not simple for obvious reasons.

\begin{thm}
Let $(M,\F)$ be a foliation  on a $C^{\infty}$ smooth manifold $M$
with no leaves of dimension $0$. Then  the commutator subgroup
$$\D=[\diff^r_c(M,\F)_0,\diff^r_c(M,\F)_0]$$ is simple if and only
if $\F_{\D}$ does not have any proper minimal set. In particular,
if $\F$ is regular, and $1\leq r\leq \dim\F$ or $r=\infty$, then
$\diff^r_c(M,\F)_0$ is simple if and only if $\F$ has no proper
minimal sets.
\end{thm}

In the proof of Theorem 1.1 in sect. 3  some ideas from Ling
\cite{li} are in use.

\section{Foliations correspond to a subclass of the class of diffeomorphism groups}

Let $1\leq r\leq\infty$ and let $L$ be a~subset of
a~$C^r$-manifold $M$ endowed with a~$C^r$-differen\-tiable
structure which makes it an~immersed submanifold. Then $L$ is
\emph{weakly imbedded} if for any locally connected topological
space $N$ and a continuous map $f:N\rightarrow M$ satisfying
$f(N)\subset L$, the map $f:N\rightarrow L$ is continuous as well.
It~follows that in this case such a~differentiable structure is
unique . A \emph{foliation of class} $C^r$ is a partition
$\mathcal{F}$ of $M$ into weakly imbedded submanifolds, called
leaves, such that the following condition holds. If~$x$~belongs to
a $k$-dimensional leaf, then there is a local chart $(U,\phi)$
with $\phi(x)=0$, and $\phi(U)=V\times W$, where $V$ is open in
$\mathbb{R}^k$, and $W$ is open in $\mathbb{R}^{n-k}$, such that
if~$L\in \mathcal{F}$ then $\phi(L\cap U)=V\times l$, where
$l=\{w\in W| \phi^{-1}(0,w)\in L\}$. A foliation is called
\emph{regular} if all leaves have the same dimension.

 Sussmann (\cite{sus}) and Stefan (\cite{st0},\cite{st2}) regarded foliations
 as collections of accessible sets in the following sense.

\begin{dff}
A smooth mapping $\phi$ of a open subset of $\mathbb{R}\times M$
into $M$ is said to be a \emph{$C^r$-arrow}, $1\leq r\leq\infty$, if\\
(1) $\phi(t,\cdot)=\phi_{t}$ is a local $C^r$-diffeomorphism for
each $t$, possibly with empty domain,\\
(2) $\phi_0=\id$ on its domain,\\
(3) $\dom(\phi_t)\subset \dom(\phi_s)$ whenever $0\leq s<t$.
\end{dff}

Given an arbitrary set of arrows $\mathcal A$, let $\mathcal A^*$
be the totality of local diffeomorphisms $\psi$ such that $\psi =
\phi(t,\cdot)$ for some $\phi\in \mathcal A$, $t\in \mathbb{R}$.
Next $\hat{\mathcal A}$ denotes the set consisting of all local
diffeomorphisms being finite compositions of elements from
$\mathcal A^*$ or $(\mathcal A^*)^{-1}=\{\psi^{-1}|\psi \in
\mathcal A^*\}$, and of the identity. Then the orbits of
$\hat{\mathcal A}$ are called \emph{accessible} sets of $\mathcal
A$.

For $x\in M$ let $\mathcal A(x)$, $\bar{\mathcal A}(x)$ be the
vector subspaces of $T_x M$ generated by
$$\{\dot{\phi}(t,y)|\phi\in \mathcal A,\phi_t(y)=x\},
\quad \{d_y\psi(v)|\psi\in \hat{\mathcal A},\psi(y)=x,v\in
\mathcal A(y)\},$$ respectively. Then we have (\cite{st0})
\begin{thm}
Let $\mathcal A$ be an arbitrary set of $C^r$-arrows on M. Then
\begin{enumerate} \item every accessible set of $\mathcal A$ admits a (unique)
$C^r$-differentiable structure of~a~connected weakly imbedded
submanifold of~$M$; \item the collection of accessible sets
defines a foliation $\mathcal{F}$;  and \item
$\D(\F):=\{\bar{\mathcal A}(x)\}$ is the tangent distribution of
$\mathcal{F}$.\end{enumerate}
\end{thm}

 Let $G\leq \diff^{r}(M)$ be an isotopically connected group of
diffeomorphisms. Let $\A_G$ be the totality of restrictions of
isotopies $\R\t M\ni(t,x)\mapsto f_t(x)\in M$ in $G$ to open
subsets of $\R\t M$. Then by $\F_G$ we denote the foliation
defined by the set of arrows $\A_G$. Observe that $\hat\A_G=\A_G$
and, consequently, $\bar\A_G(x)=\A_G(x)$.

\begin{rem}
(1) Of course, any subgroup  $G\leq \diff^{r}(M)$ determines
uniquely a foliation. Namely, $G$ defines uniquely its maximal
subgroup $G_0$ which is isotopically connected.

(2) Denote by $G_c$ the subgroup of all compactly supported
elements of $G$. Then $G_c$ need not be isotopically connected
even if $G$ is so. In fact,  let $G=\diff^r(\R^n)_0$, $1\leq
r\leq\infty$. Then every $f\in G_c$ is isotopic  to the identity
but the isotopy need not be in $G_c$. That is, $G_c$ is not
isotopically connected. Observe that the exception is the $C^0$
case: due to Alexander's trick for $r=0$ (see, e.g., \cite{ed-ki},
p.70) $G_c$ is isotopically connected.

Likewise, let $C=\R\t\mathbb S^1$ be the annulus and let
$G=\diff^r(C)_0$. Then there is the twisting number epimorphism
$T:G_c\r\mathbb Z$. It is easily seen that $f\in G_c$ is joined to
id by a compactly supported isotopy iff $T(f)= 0$. Consequently,
$G_c$ is not isotopically connected.
\end{rem}

Denote by  $\frak G^r(M)$ (resp. $\frak G^r_c(M)$), $1\leq
r\leq\infty$, the totality of isotopically connected (resp.
isotopically connected through compactly supported isotopies)
groups of $C^r$ diffeomorphisms of $M$. Next the symbol $\frak
F^r(M)$ will stand for the totality of foliations of class $C^r$
on $M$. Then each $G\in\frak G^r(M)$ determines uniquely a
foliation from $\frak F^r(M)$, denoted by $\F_G$. That is, we have
the mapping $\beta_M: \frak G^r(M)\ni G\mapsto \F_G\in\frak
F^r(M)$. Conversely, to any foliation $\F\in\frak F^r(M)$ we
assign $G_{\F}:=\diff^r_c(M,\F)_0$ and we get the mapping
$\alpha_M:\frak F^r(M)\ni \F\mapsto G_{\F}\in \frak G^r_c(M)$. The
following is obvious.

\begin{prop}
One has $\beta_M\ci\alpha_M=\id_{\frak F^r(M)}$. In particular
$$\alpha_M:\frak F^r(M)\ni \F\mapsto
G_{\F}\in\frak G_c^r(M)$$ is an  injection identifying the class
of $C^r$-foliations with a subclass of $C^r$-diffeo\-mor\-phism
groups.
\end{prop}

 Observe that
usually $(\alpha_M\ci\beta_M)(G)\in\frak G^r_c(M)$ is not a
subgroup of $G$ even if $G\in\frak G^r_c(M)$. For instance, take
the group of Hamiltonian diffeomorphisms of a Poisson manifold,
see \cite{va}. See also examples in \cite{ry0}.

\begin{rem}
Note that we can also define $\alpha'_M:\frak F^r(M)\ni \F\mapsto
G'_{\F}\in \frak G^r(M)$, where $G'_{\F}:=\diff^r(M,\F)_0\in\frak
G^r(M)$, and we get another identification of the class of
$C^r$-foliations with a subclass of $C^r$-diffeo\-mor\-phism
groups. However we prefer $\alpha_M$ to $\alpha'_M$ because of
Proposition 2.11 below.
\end{rem}

For $\F_1,\F_2\in\frak F^r(M)$ we say that $\F_1$ is a
\emph{subfoliation} of $\F_2$ if each leaf of $\F_1$ is contained
in a leaf of $\F_2$. We then write $\F_1\prec\F_2$. By a
\emph{flag structure} we mean a finite sequence
$\F_1\prec\cdots\prec\F_k$ of foliations of $M$. Next, by the
\emph{intersection} of $\F_1, \F_2$ we mean the partition
$\F_1\bar\cap\F_2:=\{L_1\cap L_2:\, L_i\in\F_i, i=1,2\}$ of $M$.
Clearly, if $\F_1\bar\cap\F_2$ is a foliation then
$\F_1\bar\cap\F_2\prec\F_i$, $i=1,2$.

It is a rare phenomenon that $\F_1\bar\cap\F_2$ would be a regular
foliation, provided $\F_1, \F_2$ are regular. In the category of
(singular) foliations it may happen more often.

\begin{prop}
\begin{enumerate}
\item If the distribution $\D(\F_1\bar\cap\F_2)$ is of class $C^r$
(\cite {st0}) then $\F_1\bar\cap\F_2$ is a foliation. \item If
$G_1,G_2\in\frak G^r(M)$ have the intersection $G=G_1\cap G_2$
isotopically connected then $\F_G=\F_{G_1}\bar\cap\F_{G_2}$. \item
For $\F_1,\F_2\in\frak F^r(M)$, if the intersection
$\F_1\bar\cap\F_2$ is a foliation then there is $G\in\frak G^r(M)$
such that  $G\leq G_{\F_1}\cap G_{\F_2}$ and
$\F_G=\F_1\bar\cap\F_2$. \item For $\F_1,\F_2\in\frak F^r(M)$, if
$G_{\F_1}\cap G_{\F_2}$ is isotopically connected then the
intersection $\F_1\bar\cap\F_2$ is a foliation.
\end{enumerate}
\end{prop}
\begin{proof}
(1) In fact, the distribution of $\F_1\bar\cap\F_2$ is then
integrable.

(2) Denote by $\I G$ the set of all isotopies in $G$. Clearly,
$\I(G_1\cap G_2)=\I{G_1}\cap\I{G_2}$ for arbitrary $G_1,
G_2\in\frak G^r(M)$. For $x\in M$, set $\I G(x):=\{y\in M:
(\exists f\in\I G)(\exists t\in I)\, y=f_t(x)\}$. By definition,
$L_x=\I G(x)$, where $L_x\in\F_G$ is a leaf meeting $x$.
Therefore, since $ G_1, G_2, G$ are isotopically connected we have
$L_x=\I G(x)=\I G_1(x)\cap \I G_2(x)=L^1_x\cap L_x^2$, where
$L^i_x\in\F_{G_i}$, $i=1,2$.

(3) Set $\F=\F_1\bar\cap\F_2$ and $G=G_{\F}$. Use Prop. 2.4.

(4) In view of Prop. 2.4 we have $\F_{G_{\F_0}}=\F_0$ for all
$\F_0\in\frak F^r(M)$. Put $G=G_{\F_1}\cap G_{\F_2}$. Therefore,
in view of (2),
$\F_1\bar\cap\F_2=\F_{G_{\F_1}}\bar\cap\F_{G_{\F_2}}=\F_G$ is a
foliation.
\end{proof}

Let  $\F_1\prec\cdots\prec\F_k$ be a flag structure on $M$ and let
$x\in M$. If $x\in L_i\in \F_i$ we write
 $p_i(x)=\dim L_i$, $\bar p_i(x)=p_i(x)-p_{i-1}(x)$ $(i=2\ld k)$ and
 $q_i(x)=m-p_i(x)$.

\begin{dff} A chart $(U,\phi)$ of  $M$ with $\phi(0)=x$ is called a {\it
distinguished chart} at $x$ with respect to
$\F_1\prec\cdots\prec\F_k$ if $U=V_1\t\cdots\t V_k \t W$ such that
$V_1\s \R^{p_1(x)}$,
 $V_i\s\R^{\bar p_i(x)}$ $(i\geq 2)$
and $W\s\R^{q_k(x)}$ are open balls and for any  $L_i\in \F_i$ we
have
$$
\phi(U)\cap L_i=\phi (V_1\t\cdots\t V_i\t l_i), $$ where
$l_i=\{w\in V_{i+1}\t\cdots\t V_k\t W:\ \phi(0,w)\in L_i\}$ for
$i=1\ld k$.
\end{dff}
Observe that actually the above $\phi$ is an inverse chart;
following \cite{st2} we call it a chart for simplicity. Notice as
well that in the above definition one need not assume that $\F_i$
is a foliation but only that it is a partition by weakly imbedded
submanifolds; that $\F_i$ is a foliation follows then by
definition.

\begin{thm}  Let  $G_1\leq\ldots\leq G_k\leq\diff^r(M)$ be an increasing sequence of
diffeomorphism groups of $M$. Then
$\F_{G_1}\prec\cdots\prec\F_{G_k}$ admits a distinguished chart at
any $x\in M$. \end{thm}

In fact, it is a straightforward consequence of Theorem 2 in
\cite{ry0}.

\begin{cor} Let   $G_1\leq\ldots\leq G_k\leq\diff^r(M)$ and let
$(L,\sigma)$ be a leaf of $\F_{G_k}$. Then all $G_i$ preserve $L$,
and $\F_{G_1|L}\prec\cdots\prec\F_{G_{k-1}|L}$ is a flag structure
on $L$. Moreover, a distinguished chart at $x$ for
$\F_{G_1|L}\prec\cdots\prec\F_{G_{k-1}|L}$ is the restriction to
$L$ of a distinguished chart at $x$ for
$\F_{G_1}\prec\cdots\prec\F_{G_k}$.
\end{cor}

The following property of paracompact spaces is well-known.
\begin{lem}
If $X$ is a paracompact space and $\U$ is an open cover of $X$,
then there exists an open cover $\V$ starwise finer than $\U$,
that is
 for all $V\in \V$  there
is $U\in\U$ such that $\st^{\V}(V)\s U$. Here
$\st^{\V}(V):=\bigcup\{V'\in\V:\; V'\cap V\neq\emptyset\}$. In
particular, for all $V_1, V_2\in \V$ with $V_1\cap
V_2\neq\emptyset$ there is $U\in\U$ such that $V_1\cup V_2\subset
U$.
\end{lem}

\begin{prop}
If $\F\in\frak F^r(M)$ then $G_{\F}=\a(\F)$ is factorizable.
\end{prop}
\begin{proof}
Let $\frak X_c(M,\F)$ be the Lie algebra of all compactly
supported vector fields on $M$ tangent to $\F$. Then there is a
one-to-one correspondence between isotopies $f_t$ in $G_{\F}$ and
smooth paths $X_t$ in $\frak X_c(M,\F)$ given by the equation
\begin{equation*}
\frac{df_t}{dt}=X_t\circ f_t \quad \mathrm{with} \quad f_0=\id.
\end{equation*}
Let $f=(f_t)\in\I G_{\F}$  and let $X_t$ be the corresponding
family in $\frak{X}_c(M,\F)$. By considering
$f_{(p/m)t}f^{-1}_{(p-1/m)t}$, $p=1,\dots ,m$, instead of $f_t$ we
may assume that $f_t$ is close to the identity.

Let $\U$ be an open cover of $M$.
 We choose a
 family of open sets, $(V_j)_{j=1}^s$, which
is starwise finer than $\U$, and satisfies $\supp(f_t)\subset
V_1\cup\dots \cup V_s$ for each $t$. Let $(\lambda_j)_{j=1}^s$ be
a partition of unity subordinate to $(V_j)$, and let
$Y_t^j=\lambda_jX_t$. We set
$$X_t^j=Y_t^1+\dots +Y_t^j,\quad j=1,\dots ,s,$$
and $X_t^0=0$. Each of the smooth families $X_t^j$ integrates to
an isotopy $g_t^j$ with support in $V_1\cup\dots \cup V_j$. We get
the fragmentation
$$f_t=g_t^s=f_t^s\circ \dots \circ f_t^1,$$
where $f_t^j=g_t^j\circ (g_t^{j-1})^{-1}$, with the required
inclusions
$$\supp(f_t^j)=\supp(g_t^j\circ (g_t^{j-1})^{-1})\subset \st(V_j)\subset U_{i(j)}$$
which hold if $f_t$ is sufficiently small. Thus the group of
isotopies of $G_{\F}$ is factorizable. Consequently, $G_{\F}$
itself is factorizable.
\end{proof}

\begin{rem} The identification $\a_M$ enables us to consider several
new properties of foliations from $\frak F^r(M)$. For instance,
one can say that a foliation $\F$ is \emph{perfect} if so is the
corresponding diffeomorphisms  group $G_{\F}=\a_M(\F)$. As we
mentioned before it is known that $G_{\F}=\diff^r_c(M,\mathcal
F)_0$ is perfect provided $\F$ is regular and
 $1\leq r\leq\dim\F$ or $r=\infty$ (\cite{ry1}, \cite{Tsu1}, \cite{ry2}).
 It is not known whether $G_{\F}$ is perfect for singular foliations and
 a possible proof seems to be very difficult. In turn, a possible perfectness of $G_{\F}=\diff^r_c(M,\mathcal
F)_0$ with $r$ large is closely related to the simplicity of
$\diff^{n+1}_c(M^n)_0$, see \cite {le-ry}.

Likewise, one can consider \emph{uniformly perfect} or
\emph{bounded} foliations by using the corresponding notions for
groups, see \cite{ry6} and references therein.
\end{rem}

Finally consider the following important feature of subclasses of
the class $\frak F^r(M)$, depending also on $M$ and $r$. A
subclass $\frak K$ of $\frak F^r(M)$ is called \emph{faithful} if
the following holds: For all $\F_1,\F_2\in\frak K$ and for any
group isomorphism $\Phi:\alpha_M(\F_1)\cong\alpha_M(\F_2)$ there
is a $C^r$ foliated diffeomorphism $\phi:(M,\F_1)\cong(M,\F_2)$
such that $\forall f\in\alpha_M(\F_1)$, $\Phi(f)=\phi\ci
f\ci\phi^{-1}$. From reconstruction results of Rybicki \cite{ry4}
and Rubin \cite{Rub} it is known that the class of regular
foliations of class $C^{\infty}$, $\frak F_{reg}^{\infty}(M)$, is
faithful.

\section{Proof of Theorem 1.1 and 1.2}

\emph{Proof of Theorem 1.1}. First observe that the fact  that a
foliation $\F$ has no proper minimal set is equivalent to the
statement that all leaves of $\F$ are dense.

$(\Rightarrow)$ Assume that $\emptyset\neq L\s M$ is a proper
closed saturated subset of $M$. Choose $x\in M\setminus L$. We
prove the following statement:

$(*)$ there are a ball $U\s M\setminus L$ with $x\in U$ and $g\in
[G_U,G_U]$ such that  $g(x)\neq x$.

We are done in view of $(*)$ by setting $H:=\{g\in [G,G]:
g|_L=\id_L\}$. To prove $(*)$, choose balls $U$ and $V$ in $M$
such that $x\in V\s \overline{V}\s U$. Take $f\in G$ such that
$f(x)\neq x$. In light of the assumption, for $\U=\{U,
\setminus\overline V\}$
 we may write $g=g_r\ldots g_1$,
where all $g_i$ are supported in elements of $\U$. Let
$s:=\min\{i\in\{1\ld r\}:\;  \supp(g_i)\s U \text{\; and\;}
g_i(x)\neq x\}$. Then $g_s\in G_U$ satisfies $g_s(x)\neq x$.

Now take an open $W$ such that $x\in W\s U$ and $g_s(x)\not\in W$.
Choose $f\in G_W$ with $f(x)\neq x$ by an argument similar to the
above. It follows that $f(g_s(x))=g_s(x)\neq g_s(f(x))$ and,
therefore, $[f,g_s](x)\neq x$. Thus $g=[f,g_s]$ satisfies the
claim.

$(\Leftarrow)$ First observe the following commutator formulae for
all $f,g,h\in G$
\begin{equation}
[fg,h]=f[g,h]f^{-1}[f,h],\quad [f,gh]=[f,g]g[f,h]g^{-1}.
\end{equation}
Next, in view of a theorem of Ling \cite{li} we have that $[G,G]$
is a perfect group, that is
\begin{equation}
[G,G]=[[G,G],[G,G]].
\end{equation}

Suppose that $H$ is a nontrivial normal subgroup of $[G,G]$. Let
$x\in M$ satisfy $h(x)\neq x$ for some $h\in H$. Fix a ball $U_0$
such that $h(U_0)\cap U_0=\emptyset$. By the definition of
$\F_{[G,G]}$ and the assumption that each leaf $L\in\F_{[G,G]}$ is
dense, for every $y\in M$ there are a ball $U_y$ with $y\in U_y$
and $f_y\in [G,G]$ such that $f_y(U_y)\s U$. Let $\U=\{U_y\}_{y\in
M}$.

Due to Lemma 2.10  we can find an open cover $\V$  starwise finer
than $\U$. We denote $\U^G=\{g(U)|\, U\in\U,\, g\in[G,G]\}$ and
\begin{equation*}G^{\U}=\prod\limits_{U\in \U^G}[G_U,G_U].\end{equation*}
 By assumption  $G$ is factorizable with respect to $\V$.
 First we show that
$[G,G]\s G^{\U}$, i.e. that any $[g_1,g_2]\in[G,G]$ can be
expressed as a product   elements of $G^{\U}$ of the form
$[h_1,h_2]$, where $h_1,h_2\in G_U$ for some $U\in\U^G$. In view
of (3.1) and (3.2) we may assume that $g_1, g_2\in[G,G]$. Now the
relation $[G,G]\s G^{\U}$  is an immediate consequence of (3.1)
and the fact that $\V$ is starwise finer than $\U$.

  Next we have to show that $G^{\U}\s H$. It suffices to check that for
every $f,g\in G_U$ with $U\in \U$ the bracket $[f,g]$ belongs to
$H$. This implies that for every $f,g\in G_U$ with $U\in\U^G$ one
has $[f,g]\in H$, since $H$ is a normal subgroup in $[G,G]$.

We have fixed $h\in H$ and $U_0$ such that $h(U_0)\cap
U_0=\emptyset$. If $U\in\U$ and $f,g\in G_U$, take $k\in[G,G]$
such that $k(U)\s U_0$, and put $\bar f=kfk^{-1}$, $\bar
g=kgk^{-1}$. It follows that $[h\bar fh^{-1}, \bar g]=\id$.
Therefore, $[\bar f,\bar g]=[[h,\bar f],\bar g]\in H$, and we have
also that $[f,g]\in H$. Thus we have $G^{\U}\s H$ and,
consequently, $[G,G]\leq H$, as required. \quad $\square$

\medskip

\emph{Proof of Theorem 1.2}. By assumption and Prop. 2.11
$\diff^r(M,\mathcal F)_0$ is factorizable and non-fixing. Since
$\diff^r(M,\mathcal F)_0$ is isotopically connected, the first
assertion follows from Theorem 1.1. The second assertion is a
consequence of theorems stating that $\diff^r(M,\mathcal F)_0$ is
perfect (\cite{ry1} and \cite{Tsu1} for $r=\infty$, and \cite{Mat}
and \cite{ry2} for $1\leq r\leq\dim\F$). \quad $\square$

\end{document}